\documentclass[11pt]{amsart}
\usepackage{amsmath}
\usepackage[dvips]{epsfig}
\usepackage{xcolor}
\theoremstyle{plain}
\newtheorem{theorem}{Theorem}[section]
\newtheorem{lemma}[theorem]{Lemma}

\newtheorem{cor}[theorem]{Corollary}

\usepackage{enumitem}
\theoremstyle{definition}
\newtheorem{remark}[theorem]{Remark}

\numberwithin{equation}{section}

\begin{document}

\title[A Minkowski inequality for Horowitz-Myers geon]
{A Minkowski inequality for Horowitz-Myers geon}

\author[Alaee]{Aghil Alaee}
\address{\parbox{\linewidth}{Aghil Alaee\\
		Department of Mathematics and Computer Science, Clark University, Worcester, MA 01610, USA\\
		Center of Mathematical Sciences and Applications, Harvard University, Cambridge, MA 02138, USA}}
\email{aalaeekhangha@clarku.edu, aghil.alaee@cmsa.fas.harvard.edu}
\author[Hung]{Pei-Ken Hung}
\address{\parbox{\linewidth}{Pei-Ken Hung\\
		Department of Mathematics, Massachusetts Institute of Technology, Cambridge, MA 02139, USA}}
\email{pkhung@mit.edu}
\thanks{A. Alaee acknowledges the support of the AMS-Simons travel grant. }
\date{\today}
\begin{abstract}We prove a sharp inequality for toroidal hypersurfaces in three and four dimensional
 Horowitz-Myers geon. This extend previous results on Minkowski inequality in the static spacetime to toroidal surfaces in asymptotically hyperbolic manifold with flat
 toroidal conformal infinity.
\end{abstract}
\maketitle

\section{Introduction}
The classical {Minkowski inequality} for a closed convex hypersurface $\Sigma$ with induced metric $\gamma$ in $\mathbb{R}^n$ reads
\begin{equation}\label{minko}
\int_{\Sigma}H\, \text{dvol}_\gamma\geq |\mathbb{S}^{n-1}|^{\frac{1}{n-1}}|\Sigma|^{\frac{n-2}{n-1}}.
\end{equation}Moreover, the rigidity holds if and only if $\Sigma$ is isometric to the $(n-1)$-sphere. The inequality \eqref{minko} has been improved to include star-shaped surfaces \cite{Guan1,Guan2}. One can replace the star-shaped condition by an outward minimizing condition using the weak inverse mean curvature flow of Huisken and Ilmanen \cite{HuiskenIlmanen}. It is an open problem whether \eqref{minko} holds for general mean convex
hypersurfaces. This inequality has been generalized to different settings. In \cite{BrendleHuanWang}, Brendle, Wang and the second author extended this inequality to convex, star-shaped hypersurface $\Sigma$ with induced metric $\gamma$ in the Anti-de-Sitter Schwarzschild space with horizon $\Sigma_{H}$ and static potential $\phi$
\begin{equation}
\int_{\Sigma}\phi H\, \text{dvol}_\gamma-n(n-1)\int_{\Omega}\phi\, \text{dvol}_g\geq (n-1)|\mathbb{S}^{n-1}|^{\frac{1}{n-1}}\left(|\Sigma|^{\frac{n-2}{n-1}}-|\Sigma_H|^{\frac{n-2}{n-1}}\right),
\end{equation}where $\Omega$ is a bounded region with boundary $\partial\Omega=\Sigma\cup\Sigma_H$. Moreover, the equality holds if and only if $\Sigma$ is a coordinate sphere. For bounded region with outward minimizing boundary
$\Sigma$ in the Schwarzschild space with mass $m$, Wei \cite{YWei} proved that
\begin{equation}\label{eq1.3}
\frac{1}{(n-1)|\mathbb{S}^{n-1}|}\int_{\Sigma}\phi H\,\text{dvol}_\gamma\geq \left(\frac{|\Sigma|}{|\mathbb{S}^{n-1}|}\right)^{\frac{n-2}{n-1}}-2m.
\end{equation}with rigidity for coordinate sphere. Recently, McCormick \cite{McCormickM} showed that the inequality \eqref{eq1.3} holds for asymptotically flat static spacetime of dimension $3\leq n\leq 7$. Note that all of these inequalities are for static manifolds with spherical infinity. 

In 1999, Horowitz and Myers \cite{HM} discovered a complete static spacetime $(N^{n+1},-\phi^2dt^2+g)$ with negative cosmological constant $-n$, that is a solution of the following static equation 
\begin{equation}\label{static}
g\Delta_g\phi+\phi\text{Ric}_g-\text{Hess}_g\phi=0, \qquad R_g=-n(n-1)
\end{equation}where $\text{Ric}_g$, $R_g$, $\Delta_g$, and $\text{Hess}_g$ are Ricci curvature, scalar curvature, Laplace-Beltrami operator, and the hessian with respect to metric $g$, respectively. The time constant slice of this spacetime is an asymptotically hyperbolic manifold with flat
toroidal conformal infinity and it is called \emph{the Horowitz-Myers geon} $(M,g,\phi)$ \cite{Ewoolgar}. The metric takes the form
\begin{equation}
g=ds^2+\left(\frac{d\phi}{ds}\right)^2d\xi^2+\phi^2\sum\limits_{i=3}^n(d\theta^i)^2,\qquad \phi(s):=\cosh^{\frac{2}{n}}\left(\frac{ns}{2}\right),
\end{equation}where $s\in [0,\infty)$ is the geodesic coordinate from central torus, $\xi\in [0,4\pi/n]$, and $\theta^i\in
[0,a_i]$ for $0<a_3\le\dots\le a_n$. The period of $\xi$ is chosen such that, after a coordinate change, the metric is smooth up to the central torus $\{s=0\}$. The function $\phi(s)$ satisfies the equation
	\begin{equation}\label{equ:phi}
	\frac{d\phi}{ ds} =\phi (1-\phi^{-n})^{1/2}. 
	\end{equation} 
	The total mass of the Horowitz-Myers geon is
\begin{equation}\label{def:mass}
m=-\frac{4\pi}{n}\prod_{i=3}^na_i\,.
\end{equation}They postulated that this may be the ground state for a conjectural nonsupersymmetric AdS/CFT correspondence \cite{HM}. More precisely, let $(M,g)$ be an asymptotically hyperbolic manifold with flat toroidal conformal infinity and scalar curvature $R(g)\geq -n(n-1)$. Then the total mass is at least equal to the mass of the Horowitz-Myers geon. Moreover, the rigidity holds if and only if $(M,g)$ is isometric to the Horowitz-Myers geon. The Horowitz-Myers conjecture follows from Conjecture 3 of \cite{HM}. Progress on the Horowitz-Myers conjecture has been very limited. For small perturbations of Horowitz-Myers the conjecture is true by Constable and Myers \cite{CostableMyers} and Woolgar proved the rigidity of this conjecture in 3 dimensions \cite{Ewoolgar}. Recently, in \cite{BCHMN},  there are more supporting evidence in the validity of this conjecture.  

In this paper, we investigate a Minkowski type inequality for toroidal hypersurfaces $\Sigma$ with induced metric $\gamma$ in the Horowitz-Myers geon $(M,g,\phi)$.  We consider a Minkowski-type quantity
\begin{align}\label{def:Q}
Q(\Sigma):=n(n-1)\int_{\Omega}\phi\, \text{dvol}_g-\int_{\Sigma}\phi H\, \text{dvol}_\gamma.
\end{align} 
Here $H$ is the mean curvature of $\Sigma$ in $M$ and $\Omega$ is the bonded region enclosed by $\Sigma$. Applying the Divergence Theorem and the static equation \eqref{static}, $Q(\Sigma)$ can be rewritten as
\begin{align}\label{def:Q2}
Q(\Sigma)=\int_{\Sigma}(n-1)g(\nabla\phi,\nu)- \phi H\, \text{dvol}_\gamma.
\end{align} 
We aim to prove that under certain convexity assumption,
\begin{equation}\label{mainresult}
Q(\Sigma)\leq -2^{-1}nm,
\end{equation}
where $m$ is the mass defined in \eqref{def:mass}. The main results are
\begin{theorem}\label{thm:3d}
	Let $n=3$ and $\Sigma\subset M$ be a graph over $T^{2}$. Suppose that the second fundamental form $\tilde{h}_{ab}$ of $\Sigma$ that corresponds to the conformal metric $\tilde{g}=\phi^2 g$ is non-negative definite and that
	\begin{equation}\label{eqassumption3D}
	\min_{\Sigma}\phi  >1.
	\end{equation} Then \eqref{mainresult} holds and the equality is achieved if and only if $\Sigma$ is given by a coordinate torus $s\equiv\textup{constant}$.
\end{theorem}
\begin{theorem}\label{thm:4d}
	Let $n=4$ and $\Sigma\subset M$ be a graph over $T^{3}$. Assume $\Sigma$ is symmetric along the $\xi$ direction.  Suppose that the second fundamental form $\bar{h}_{ab}$  of $\Sigma$ that corresponds to the conformal metric $\bar{g}=\left(\phi\frac{d\phi}{ds} \right)^2g$ is non-negative definite and that
	\begin{equation}\label{eqassumption}
	\min_{\Sigma}\phi^4\geq 1+\frac{2}{\sqrt{3}}.
	\end{equation}Then \eqref{mainresult} holds and the equality is achieved if and only if $\Sigma$ is given by a coordinate torus $s\equiv\textup{constant}$.
\end{theorem}	
\begin{remark}
The assumption \eqref{eqassumption3D} in Theorem \ref{thm:3d} is to avoid the coordinate singularity at $\{s\equiv 0\}.$ The assumption \eqref{eqassumption} in Theorem \ref{thm:4d} is to ensure the flow \eqref{normalflow_4D} admits a long-time solution. As long as the solution of \eqref{normalflow_4D} exists for all time, the result of Theorem \ref{thm:4d} holds. See Lemma \ref{lem:limitQ_4D} for more detail.
\end{remark}
The structure of this article is as follows. In Section \ref{Sec:graphs} we discuss the geometry of graphs and prove the monotinicity of the $Q$ in \eqref{def:Q} under a weighted normal flow.  In Section \ref{Sec:3D}, we prove the global existence of the flow in dimensions 3 and prove  Theorem \ref{thm:3d}. Finally, in Section \ref{Sec:4D}, we prove the global existence of the flow in dimensions 4 for axially symmetric graphs and prove Theorem \ref{thm:4d}.
\section{geometry of graphs}\label{Sec:graphs}
In this section, we investigate the geometry of graph $\Sigma$ in $(M,g,\phi)$. First, we define the function 
\begin{align}\label{equ:q}
q(s):=\int_0 ^s \frac{ds'}{\phi(s')}.
\end{align}
Using $q$ as a coordinate, the conformal metric $g':=\phi^{-2}g$ is asymptotic to a flat metric on $\mathbb{R}\times T^{n-1}$ 
\begin{align}\label{def:gp}
g'=dq^2+\left(\phi^{-1}\frac{d\phi}{ds}  \right) ^2d\xi^2+ \sum\limits_{i=3}^n(d\theta^i)^2.
\end{align}
We will further consider other conformal metrics
\begin{equation}\label{def:gt}
\tilde{g}=\phi^2 g,\ \bar{g}=\left(\phi\frac{d\phi}{ds} \right)^2g.
\end{equation}
The flows we consider, \eqref{normalflow} and \eqref{normalflow_4D}, are unit normal flows in $\tilde{g}$ and $\bar{g}$ respectively. \\

We use $\nabla$ and $R_{\alpha\beta\lambda\mu}$ to denote the Levi-Civita connection and the Riemannian curvature tensor of $g$ respectively. For a hypersurface $\Sigma$, we use $\{x^a\}, a=1,2,\dots , n-1$ as a local coordinate. The induced metric and the second fundamental form are denoted by $\gamma_{ab}$ and $h_{ab}$ respectively. We write $D$ for the Levi-Civita connection of $\gamma$. We use $d\textup{vol}_g$ and $d\textup{vol}_\gamma$ to denote the volume form of $g$ and $\gamma$ respectively. The tensors and connections induced by the conformal metrics $g',\tilde{g}$ or $\bar{g}$ will carry the corresponding accents. For instance, the second fundamental form  with respect to $g'$ is written as $h'_{ab}$. Also, the index in Weingarten tensor is raised accordingly. For example, $\tilde{h}_a^b=\tilde{\gamma}^{bc}\tilde{h}_{ca}$.\\

In the Appendix \ref{sec:app1}, we calculate the curvatures and the second fundamental forms with respect to the metric
$$dq^2+\Psi(q)^2d\xi^2+\sum_{i=3}^n (d\theta^i)^2,$$
for a general positive function $\Psi(q)$. The metric $g'$ corresponds to $\Psi=\phi^{-1}\frac{d\phi}{ds}$, which is the only case we will use. We further record the curvatures and the second fundamental forms under a conformal transformation $\check{g}=e^{2\psi}g'$. The metrics $g$, $\tilde{g}$ and $\bar{g}$ correspond to $\psi=\log\phi$, $ 2\log\phi$ and $ 2\log\phi+\log \frac{d\phi}{ds}$ respectively. \hfill\\


Let $\Sigma$ be a graph given by $s=v(\xi,\theta^i)$ for some smooth function $v$ defined on $T^{n-1}$. Define the height function in the $q$ coordinate as $$u(\xi,\theta^i)=q(  v(\xi,\theta^i) ).$$ Define the slope of $\Sigma$ by
\begin{align}\label{v}
\rho^2= 1+\left(\phi \bigg/  \frac{d\phi}{ds}   \right)^2\left(\frac{\partial u}{\partial\xi} \right)^2+\delta^{ij}\frac{\partial u}{\partial\theta^i}\frac{\partial u}{\partial\theta^j}. 
\end{align} 
Let $H$ be the mean curvature of $\Sigma$ with respect to $g$. A direct computation using \eqref{H-gt} with $\Psi=\phi^{-1}\frac{d\phi }{ds}$, $\psi=\log\phi$ and \eqref{equ:phi} shows 
\begin{align*}
H=  -\rho\phi^{-1}\Delta' u+(2^{-1}n)\rho^{-1} \phi^{1-n} \left(  \frac{d\phi}{ds} \right)^{-1}\delta^{ij} \frac{\partial u}{\partial\theta^i}\frac{\partial u}{\partial\theta^j} +(n-1)\rho^{-1}\phi^{-1} \frac{d \phi}{ ds	}  +(2^{-1}n)\rho^{-1} \phi^{1-n} \left( \frac{d\phi}{ds} \right)^{-1}.     
\end{align*}
Here $\Delta'$ is the Laplace-Beltrami operator of $\gamma'$. Together with $$d\textup{vol}_\gamma= \rho\phi^{n-2}\frac{d\phi}{ ds} \, d\xi\wedge d\theta^3\wedge\dots \wedge d\theta^n,$$
we derive from \eqref{def:Q2} that
\begin{equation}\label{limit}
\begin{split}
&Q(\Sigma)+ 2^{-1}nm = \int_{T^{n-1} }  \rho^2\phi^{n-2}\frac{d\phi}{ds}\Delta' u-(2^{-1}n)\delta^{ij}\frac{\partial u}{\partial\theta^i}\frac{\partial u}{\partial\theta^j} \, d\xi\wedge d\theta^3\wedge\dots \wedge d\theta^n.
\end{split}
\end{equation}

Let $p$ be a function defined on $M$ to be determined. Let $F:[0,T_0)\times T^{n-1}\to M$ be a family of embeddings that satisfies 
\begin{equation}\label{normalflow_Both}
\frac{\partial F}{\partial t} =p \nu,
\end{equation}
where $\nu$ is the unit outward normal. Denote by $\Sigma_t$ the image of $F(t,\cdot)$ and by $\gamma_t$ the induced metric. 

\begin{lemma}\label{lem:monotone}
	\hfill
	\begin{itemize}
		\item For $n=3$, by taking $p=\phi^{-1}$, $Q(\Sigma_t)$ is monotone non-decreasing along the flow \eqref{normalflow_Both}.
		\item For $n=4$, assume $\Sigma_0$ is symmetric along the $\xi$ direction. Then by taking $p=(\phi\frac{d\phi}{ds})^{-1}$, $Q(\Sigma_t)$ is monotone non-decreasing along the flow \eqref{normalflow_Both}. 
	\end{itemize}
	 Furthermore, in the two cases above, $Q(\Sigma_t)$ is strictly increasing unless $\Sigma_0$ is a coordinate torus. 
\end{lemma}

\begin{proof}
	We compute
	\begin{equation}
	\frac{d}{dt}Q(\Sigma_t )=n(n-1)\int_{\Sigma_t}p\phi\, d\textup{vol}_{\gamma_t} -\int_{\Sigma_t}\left(\frac{\partial\phi}{\partial t}  H+p\phi H^2  +  \phi\frac{\partial H}{\partial t} \right)\, d\textup{vol}_{\gamma_t}.
	\end{equation}
	Using the first variation of area and the Gauss equation,
	\begin{equation}
	\frac{\partial H}{\partial t} =-\Delta_{\gamma_t}p-p|h|^2-p \text{Ric}(\nu,\nu),  
	\end{equation}
	\begin{equation}
	n(n-1)-H^2+|h|^2=-R_{\gamma_t}-2\text{Ric}(\nu,\nu).
	\end{equation} Hence 
	\begin{equation}
	\frac{d}{dt}Q(\Sigma_t)=\int_{\Sigma_t} p\phi \left[-R_{\gamma_t}-2\text{Ric}(\nu,\nu)\right]-p Hg( \nabla \phi,\nu)+p\phi\text{Ric}(\nu,\nu)+\phi \Delta_{\gamma_t}p  \, d\textup{vol}_{\gamma_t}.
	\end{equation}
	Using the static equation \eqref{static},
	\begin{equation}
	\Delta_{\gamma_t}\phi=\Delta_{g}\phi-\nabla^2\phi(\nu,\nu)-H g(\nabla \phi,\nu)=-\phi\text{Ric}(\nu,\nu)-Hg(\nabla \phi,\nu),
	\end{equation} 
	we obtain
	\begin{equation}
	\begin{split}
	\frac{d}{dt}Q(\Sigma_t)&=\int_{\Sigma_t} -p \phi R_{\gamma_t}+p \Delta_{\gamma_t}\phi+\phi \Delta_{\gamma_t}p \, d\textup{vol}_{\gamma_t} \\
	&=\int_{\Sigma_t} -p\phi R_{\gamma_t}-2\gamma_t(D\phi ,Dp) \, d\textup{vol}_{\gamma_t}.
	\end{split}
	\end{equation}
	
	For $n=3$, by choosing $p=\phi^{-1}$, we obtain from the Gauss-Bonnet Theorem that 
	\begin{equation}
	\frac{d}{dt}Q(\Sigma_t)=\int_{\Sigma_t} -R_{\gamma_t}+2\phi^{-2}\gamma_t (D\phi,D\phi) \, d\textup{vol}_{\gamma_t} \geq 0.
	\end{equation}
 Moreover, $\frac{d}{dt}Q(\Sigma_t)=0$ if and only if $\phi$ is a constant on $\Sigma_t$. This implies $\Sigma_t$ is a coordinate torus. By reversing the flow \eqref{normalflow_Both}, $\Sigma_0$ is also a coordinate torus.\\

	For $n=4$, we assume that $\Sigma_0$ is symmetric along the $\xi$ direction. Hence $\Sigma_t$ preserves the same symmetry. Denote by $\mathring{\gamma}_t$ the induced metric on $\mathring{\Sigma}_t=\Sigma_t\cap\{\xi\equiv\textup{constant}\}$. Then
	$$\gamma_t=\mathring{\gamma}_t+\left(\frac{d\phi}{ds} \right)^2 d\xi^2.$$ 
	The scalar curvature of $\gamma_t$ and $\mathring{\gamma}_t$ are related through $$R_{\gamma_t}=R_{\mathring{\gamma}_t}-2\left(\frac{d\phi}{ds} \right)^{-1}\Delta_{\mathring{\gamma}_t}\left(\frac{d\phi}{ds} \right).$$ Assuming $p$ is also symmetric along the $\xi$ direction, we get
	$$\frac{d}{dt}Q(\Sigma_t)=\frac{4\pi}{n} \int_{\mathring{\Sigma}_t} -p\phi \frac{d\phi}{ds}     R_{\mathring{\gamma}_t}+2p\phi    \Delta_{\mathring{\gamma}_t}\left(\frac{d\phi}{ds} \right) -2 \frac{d\phi}{ds}    \mathring{\gamma}_t(D\phi ,Dp) \, d\textup{vol}_{\mathring{\gamma}_t}.$$ Taking $p=\left(\phi\frac{d\phi}{ds} \right)^{-1}$, we deduce from the Gauss-Bonnet Theorem that
	\begin{align*}
	\frac{d}{dt}Q(\Sigma_t)=&\frac{4\pi}{n} \int_{\mathring{\Sigma}_t} -R_{\mathring{\gamma}_t}+2\left( \frac{d\phi}{ds} \right)^{-1}\Delta_{\mathring{\gamma}_t}\left( \frac{d\phi}{ds} \right)-2  \frac{d\phi}{ds}   \mathring{\gamma}_t \left(D\phi ,D\left(\phi\frac{d\phi}{ds} \right)^{-1}\right)    \, d\textup{vol}_{\mathring{\gamma}_t}\\
	= &\frac{4\pi}{n} \int_{\mathring{\Sigma}_t} \left( -2\frac{d}{ds}\left( \frac{d\phi}{ds} \right)^{-1}   \frac{d^2\phi}{ds^2}  -2\left( \frac{d\phi}{ds} \right)^2 \frac{d}{ds}\left(\phi \frac{d\phi}{ds} \right)^{-1} \right)\cdot \mathring{\gamma}_t(Dv_t,Dv_t)     \, d\textup{vol}_{\mathring{\gamma}_t}.
	\end{align*}   
From \eqref{equ:phi} and $n=4$, we derive
	\begin{align*}
	-2\frac{d}{ds}\left( \frac{d\phi}{ds} \right)^{-1}\frac{d^2\phi}{ds^2}-2\left( \frac{d\phi}{ds} \right)^2 \frac{d}{ds}\left(\phi \frac{d\phi}{ds} \right)^{-1} =(1-\phi^{-4})^{-1}(6+2\phi^{-8})\geq 0.
	\end{align*}
Thus $Q(\Sigma_t)$ is monotone non-decreasing along the flow. Using the same argument for the $3$ dimensional case, $\frac{d}{dt}Q(\Sigma_t)=0$ if and only if $\Sigma_0$ is a coordinate sphere. 
\end{proof}

\section{Three Dimensional Case}\label{Sec:3D}
In this section, we fix $n=3$ and prove Theorem \ref{thm:3d}. Recall that $F:[0,T_0)\times T^{n-1}\to M$ solves
\begin{equation}\label{normalflow}
\frac{\partial F}{\partial t} =\phi^{-1} \nu.
\end{equation} 
Here we take $T_0\in (0,\infty]$ to be the largest number such that \eqref{normalflow} has a smooth graphical solution in $t\in[0,T_0).$\\

Equation \eqref{normalflow} can be viewed as the unit normal flow with respect to the conformal metric $\tilde{g}=\phi^2g $ . Recall that $\tilde{h}_{ab}$ is the the second fundamental form with respect to $\tilde{g}$. Our main assumption is  
\begin{align}\label{convex}
\tilde{h}_{ab}\ \textup{is non-negative definite for } \Sigma_0\ \textup{and}\ \min_{\Sigma_0}\phi > 1.
\end{align} 
We list two lemmas which allow us to prove Theorem \ref{thm:3d}. 
\begin{lemma}\label{lem:limitQ_3D}
	Suppose $T_0=\infty$. Then
	\begin{align}\label{limitQ_3D}
	\limsup_{t\to\infty}Q(\Sigma_t)\leq -\frac{3m}{2}.
	\end{align} 
\end{lemma}

\begin{lemma}\label{lem:T0_3D}
	Under the assumption \eqref{convex}, we have $T_0=\infty$.
\end{lemma}

\begin{proof}[Proof of Theorem \ref{thm:3d}]
	Let $\Sigma_t$ be the solution of \eqref{normalflow} starting from $\Sigma_0=\Sigma$. Combining Lemmas \ref{lem:monotone}, \ref{lem:T0_3D} and \ref{lem:limitQ_3D},
	\begin{align*}
	Q(\Sigma)\leq \limsup_{t\to\infty}Q(\Sigma_t)\leq -\frac{3m}{2}.
	\end{align*}
	Furthermore, the equality implies $Q(\Sigma_t)$ is a constant along the flow. And Lemma \ref{lem:monotone} implies $\Sigma$ is a coordinate torus.
\end{proof}
We prove Lemmas \ref{lem:limitQ_3D} and Lemma \ref{lem:T0_3D} in the rest of this section. We adapt the convention that $C$ denotes a large constant depending on $\Sigma_0$. The value of $C$ may change from line to line. \\

Denote by $v_t(\xi,\theta)$ the height function of $\Sigma_t$ and $u_t=q(v_t)$. Then $v_t$ solves the equation
\begin{equation}\label{grapheq3D}
\frac{\partial v_t}{\partial t}=\rho \phi(v_t)^{-1}.
\end{equation} 
Here $\rho$ is the slope of $\Sigma_t$ defined in \eqref{v}. We start with the $C^0$ and $C^1$ estimates for $v_t$ and $u_t$. 
\begin{lemma}\label{lem:C00}
	There exists a constant $C $ depending on $\Sigma_0$ such that for all $t\in [0,T_0)$,
	\begin{equation}\label{C00}
	\begin{split}
	|\phi(v_t)-t|\leq& C ,\ \phi(v_t)\geq  1+C^{-1}(t+1).
	\end{split}
	\end{equation} 
\end{lemma}
\begin{proof}
	At the minimum point of $v_t$, we have $\rho=1$. Together with \eqref{equ:phi}, 
	$$\frac{d}{dt}\min_{T^2} \phi(v_t)\geq  \left(1- (\min_{T^2} \phi(v_t) )^{-3} \right)^{1/2}. $$
Let $\Phi(t,a)$ be the solution to the ODE
	\begin{align*}
	\frac{\partial \Phi}{\partial t}=(1-\Phi^{-3})^{1/2},\ \Phi(0,a) =a.
	\end{align*} 
Through the ODE comparison, $$\displaystyle\min_{T^2} \phi(v_t)\geq \Phi(t,\min_{T^2} \phi(v_0)).$$ 
From now on, we take $a=\displaystyle\min_{T^2} \phi(v_0)$. By the assumption \eqref{convex}, $a>1$.  Because $\Phi(t,a)$ is increasing in $t$, we have $\frac{\partial\Phi}{\partial t}\geq (1-a^{-3})^{1/2}$. Through integration, $\Phi(t,a)\geq (1-a^{-3})^{1/2}t+a$. This implies $\phi(v_t)\geq 1+C^{-1}(t+1)$ for $C$ large enough.  Furthermore, using the equation of $\Phi$ again, 
	\begin{align*}
	\left| \frac{\partial \Phi}{\partial t}-1 \right|\leq \frac{C}{(1+t)^3}.
	\end{align*} 
Thus $|\Phi(t,a)-t|\leq C$ and $\displaystyle\min_{T^2} \phi(v_t)\geq t-C $. With a similar argument, we can also derive $$\displaystyle\max_{T^2} \phi(v_t)\leq t+C.$$ The proof is finished.
\end{proof}

For the notation simplicity, we later denote $\phi(v_t)$ and $\frac{d\phi}{ds}(v_t)$ by $\phi$ and $\frac{d\phi}{ds} $ respectively. Recall that
\begin{align*}
\rho^2=1+\left(\frac{\partial u_t}{\partial\theta}\right)^2 +\left(\phi \bigg/  \frac{d\phi}{ds}   \right)^2\left(\frac{\partial u_t}{\partial\xi}\right)^2.
\end{align*}  

\begin{lemma}\label{lem:C11}
	There exists a constant $C$ depending on $\Sigma_0$ such that for $t\in [0,T_0)$,
	\begin{equation}\label{C11}
\max_{\Sigma_t}	(\rho^2-1 ) \leq  (C^{-1}t+1)^{-4} \max_{\Sigma_0}(\rho^2-1) .
	\end{equation}
\end{lemma}
\begin{proof}Using \eqref{equ:phi} and \eqref{grapheq3D}, we have
	\begin{align*}
	\frac{\partial^2 u_t}{\partial x^a\partial t}=&\phi^{-2}\frac{\partial \rho}{\partial x^a}-2\rho \phi^{-1}(1-\phi^{-3})^{1/2} \frac{\partial u_t}{\partial x^a}  
	\end{align*}   
	and
	\begin{align*}
	\frac{\partial }{\partial t} \left(\phi \bigg/  \frac{d\phi}{ds}   \right)^2  =&-3\rho \phi^{-4}(1-\phi^{-3})^{-3/2}.
	\end{align*}It follows that
	\begin{align*}
	\frac{\partial  }{\partial t}(\rho^2-1) =&2\phi^{-2}\left( \frac{\partial u_t}{\partial\theta}\frac{\partial \rho}{\partial\theta}+\left(\phi \bigg/  \frac{d\phi}{ds}   \right)^2   \frac{\partial u_t}{\partial\xi}\frac{\partial \rho}{\partial\xi} \right)\\
	&-3 \rho\phi^{-4}(1-\phi^{-3})^{3/2} \left(\frac{\partial u_t}{\partial\xi}\right)^2 -4\rho \phi^{-1}(1-\phi^{-3})^{1/2}  (\rho^2-1).
	\end{align*}
	At the maximum point of $ \rho^2 $, $d\rho	=0$ and
	\begin{align*}
	\frac{\partial  }{\partial t}(\rho^2-1) \leq  &-4\rho \phi^{-1}(1-\phi^{-3})^{1/2}  (\rho^2-1).
	\end{align*}
	Using \eqref{C00}, there exists a constant $C $ such that for all $t\in [0,T_0)$,
	\begin{align*}
	4\rho \phi^{-1}(1-\phi^{-3})^{1/2}\geq  4 \phi^{-1}(1-\phi^{-3})^{1/2}\geq 4(t+C)^{-1}.
	\end{align*}
	Therefore,
	\begin{align*}
	\frac{d  }{d t}\max_{\Sigma_t}(\rho^2-1) \leq  &-4(t+C)^{-1} \max_{\Sigma_t}(\rho^2-1).
	\end{align*} 
	By the ODE comparison, 
	\begin{align*}
	\max_{\Sigma_t}(\rho^2-1)  \leq  (C^{-1}t+1)^{-4} \max_{\Sigma_0}(\rho^2-1) .
	\end{align*}
	The proof is finished.
\end{proof}

Recall that $g'=\phi^{-2}g$ is a conformal metric which asymptotic to the flat metric on $\mathbb{R}\times T^2$. Let $\gamma'_t$ be the induced metric of $g'$ on $\Sigma_t$ and $D'$ be the Levi-Civita connection of $\gamma'_t$.  Lemmas \ref{lem:C00} and \ref{lem:C11} imply, provided $T_0=\infty$, $\gamma'_t$ converges the the flat metric $d\xi^2+d\theta^2$ in $C^0$. Next, we use the second fundamental form $h'_{ab}$ to bound the hessian of $u_t$ from below. 
\begin{lemma}\label{lem:C44}
	Suppose $T_0=\infty.$ Then there exists a constant $C$ depending on $\Sigma_0$ such that  for $t\geq C+1$,
	\begin{align*}
	D'_b(D')^a u_t \geq -  \left( C (t-C )^{-2}\log t\right) \delta_b^a. 
	\end{align*}
\end{lemma}

\begin{proof}
The flow \eqref{normalflow} can be rewritten as
	\begin{align*}
	\frac{\partial F }{\partial t}= \phi^{-2}\nu'.
	\end{align*} 
	Here $\nu'=\phi\nu$ is the unit normal vector with respect to $g'$.
	The evolution equation of $(h')_b^a$ is given by 
	\begin{align*}
	\frac{\partial}{\partial t}(h')_b^a=-\phi^{-2}(h')_b^c(h')_c^a-D'_b(D')^a\phi^{-2}-\phi^{-2} (dx^a)_\alpha (\nu')^\beta \left(\frac{\partial}{\partial x^b} \right)^\lambda (\nu')^\mu   ({R}')^\alpha_{\ \beta\lambda\mu}.
	\end{align*}
	Here $(R')_{\alpha\beta\lambda\mu}$ is the Riemannian curvature tensor of $g'$. In the estimates below, the norms are computed with respect to $\gamma'_t$. From \eqref{curvature-g'} and \eqref{equ:phi}, the only non-zero component of $R'$ is
	\begin{align*}
	R'_{q\xi q\xi}=3\phi^{-1}(1-\phi^{-3}).
	\end{align*} 
	Together with \eqref{C00}, for $t$ large enough,
	\begin{align*}
	\left| -\phi^{-2} (dx^a)_\alpha (\nu')^\beta \left(\frac{\partial}{\partial x^b} \right)^\lambda (\nu')^\mu   ({R}')^\alpha_{\ \beta\lambda\mu} \right|\leq C(t-C)^{-3}.
	\end{align*} 
	Using \eqref{equ:phi} and \eqref{equ:q},
	\begin{align*}
	-D'_b(D')^a\phi^{-2} =&2\phi^{-1}(1-\phi^{-3})^{1/2} D'_b(D')^a {u}_t+(-2+5\phi^{-3})D'_b {u}_t (D')^a{u}_t.
	\end{align*} 
	The relation between $h'_{ab}$ and $D'_bD'_a u_t$ is given in Corollary \ref{cor:h-hessian} with $\Psi=\phi^{-1}\frac{d\phi}{ds}$. Using \eqref{equ:phi}, we derive
	\begin{align}\label{hessianu3D}
	D'_b(D')_a {u}_t=-\rho^{-1}h'_{ab}+\frac{3}{2} \phi^{-2}(1-\phi^{-3})^{1/2}D'_b\xi D'_a\xi.
	\end{align} 
	Hence
	\begin{align*}
	-D'_b(D')^a\phi^{-2}=&-2\rho^{-1}\phi^{-1}(1-\phi^{-3})^{1/2}(h')_b^a+(-2+5\phi^{-3})D'_b {u}_t (D')^a{u}_t\\
	&+3\phi^{-3}(1-\phi^{-3})  D'_b\xi  (D')^a\xi .
	\end{align*}
	From \eqref{C00} and \eqref{C11}, for $t$ large enough, we have $$ 
	  |D'_b {u}_t (D')^a{u}_t|\leq  C(t-C)^{-4}$$ and $$ 
	   |3\phi^{-3}(1-\phi^{-3})D'_b \xi (D')^a \xi   |\leq   (t-C)^{-3}.$$ Putting the above together, we get
	   	\begin{align*}
	\left|\frac{\partial}{\partial t}(h')_b^a+2\rho^{-1}\phi^{-1}(1-\phi^{-3})^{1/2} (h')_b^a+\phi^{-2}(h')_b^c(h')_c^a\right|\leq C(t-C)^{-3}. 
	\end{align*} 
		Let $w_t$ be the maximum eigenvalue of $(h')_a^b$. Using again \eqref{C00} and \eqref{C11}, for $t$ large enough, we have $
	   2\rho^{-1}\phi^{-1}(1-\phi^{-3})^{1/2}\geq  2(t+C)^{-1}$. Therefore,
	\begin{align*}
	\frac{d w_t}{d t}\leq -2(t+C)^{-1} w_t+C(t-C)^{-3}.
	\end{align*} 
	From the ODE comparison, $w_t\leq C(t-C)^{-2}\log t $ for $t$ large enough. The assertion then follows from the view of \eqref{hessianu3D}.
\end{proof}

\begin{proof}[Proof of Lemma \ref{lem:limitQ_3D}]
	From \eqref{limit},
	\begin{align*}
	Q(\Sigma_t)+\frac{3m}{2}=&\int_{T^2 }  \rho^2\phi \frac{d\phi}{ds}\Delta' u_t- \frac{3}{2} \left(\frac{\partial u_t}{\partial\theta}\right)^2 \, d\xi\wedge d\theta. 
	\end{align*}
	 From Lemma \ref{lem:C11}, $\left(\frac{\partial u_t}{\partial \theta}\right)^2\leq C(t-C)^{-4}$. It remains to prove
	\begin{align*}
	\limsup_{t\to\infty} \int_{T^2}  \rho^2\phi \frac{d\phi}{ds}\Delta' u_t \, d\xi\wedge d\theta\leq 0. 
	\end{align*}
	Recall that $\gamma'_t$ is the induced metric of $\Sigma$ from $g'$. Because $$d\textup{vol}_{\gamma'_t}=\rho \phi^{-1}\frac{d\phi}{ds}\, d\xi\wedge d\theta,$$
	\begin{align*}
	&\int_{ {T}^{2}}  \rho^2\phi \frac{d\phi}{ds}\Delta' u_t \, d\xi\wedge d\theta=  \int_{\Sigma_t} \rho \phi^2  \Delta' u_t \, d\textup{vol}_{\gamma'_t} 
	= \underbrace{ \int_{\Sigma_t}  (\rho-1) \phi^2  \Delta' u_t \, d\textup{vol}_{\gamma'_t} }_{\textsc{I}}+\underbrace{ \int_{\Sigma_t}   \phi^2  \Delta' u_t \, d\textup{vol}_{\gamma'_t} }_{\textsc{II}}.
	\end{align*}
	From Lemmas \ref{lem:C00}, \ref{lem:C11} and \ref{lem:C44}, $\phi^2\leq C (t+C)^2$,  $|\rho-1|\leq C(t-C)^{-4}$ and
	\begin{align*}
	\int_{\Sigma_t}   |\Delta' u_t| \, d\textup{vol}_{\gamma'_t}=2\int_{\Sigma_t}   (\Delta' u_t)_- \, d\textup{vol}_{\gamma'_t}\leq C (t-C)^{-2}\log t.
	\end{align*}  
	Therefore the first term $\textsc{I}$ goes to zero. Through integration by parts,
	\begin{align*}
	\textsc{II}=-\int_{\Sigma_t} 2\phi \frac{d\phi}{ds} |D' u_t|^2_{\gamma'_t} \, d\textup{vol}_{\gamma'_t}\leq 0.
	\end{align*}
	Thus the assertion \eqref{limitQ_3D} follows.
\end{proof}

Lastly, we prove Lemma \ref{lem:T0_3D}. Recall that $\tilde{g}=\phi^2 g$ and that $\tilde{h}_b^a$ is the Weingarten tensor with respect to $\tilde{g}$. The flow \eqref{normalflow} is a unit normal flow with respect to $\tilde{g}$. The evolution equation of $\tilde{h}_b^a$ is given by
\begin{equation}\label{equ:ht}
\frac{\partial}{\partial t}\tilde{h}_b^a=-\tilde{h}_b^c\tilde{h}_c^a-(dx^a)_\alpha \tilde{\nu}^\beta \left(\frac{\partial}{\partial x^b} \right)^\lambda \tilde{\nu}^\mu  \tilde{R}^\alpha_{\ \beta\lambda\mu}.
\end{equation}
\begin{lemma}\label{lem:C22}
	Suppose the assumption \eqref{convex} holds. Then there exists a continuous function $C(t)$ defined on $[0,\infty)$ such that for all $t\in [0,T_0)$,
	\begin{align*}
	|\tilde{h}_b^a|_{\tilde{\gamma}_t}\leq |\tilde{h}_b^a|_{\tilde{\gamma}_0}+C(t) .
	\end{align*}
\end{lemma}
\begin{proof}
	Using \eqref{curvature-gt} with $\Psi=\phi^{-1}\frac{d\phi}{ds}$, $\psi=2\log\phi$ and \eqref{equ:phi}, we derive that non-zero components of $\widetilde{R}_{\alpha\beta\lambda\mu}$ are 
	\begin{align*}
	\widetilde{R}_{q\xi q\xi}=&-\phi^6(2+\phi^{-3})(1-\phi^{-3}), \\
	\widetilde{R}_{q\theta  q\theta}=&-\phi^6(2+\phi^{-3}),\\
	\widetilde{R}_{\xi\theta  \xi\theta}=&-\phi^6(4-\phi^{-3})(1-\phi^{-3}).
	\end{align*}
	In particular, the sectional curvature of $\tilde{g}$ is non-positive. From the view of \eqref{equ:ht}, $\tilde{h}_b^a$ remains non-negative definite under the assumption \eqref{convex}. From \eqref{C00}, there exists a continuous function $C(t)$ such that for any $t\in [0,T_0)$,
	\begin{align*}
	\left| (dx^a)_\alpha \tilde{\nu}^\beta \left(\frac{\partial}{\partial x^b} \right)^\lambda \tilde{\nu}^\mu  \tilde{R}^\alpha_{\ \beta\lambda\mu} \right|\leq C(t).
	\end{align*}
	Applying the ODE comparison to \eqref{equ:ht}, we have an upper bound for the maximum eigenvalue of $\tilde{h}^a_b$ in any finite time. The proof is finished. 
\end{proof}
\begin{proof}[Proof of Lemma \ref{lem:T0_3D}]
	Recall that $v_t$ is the height function of $\Sigma_t$ in the $s$ coordinate and $u_t=q(v_t)$. It suffices to show that for any $T <\infty$, there exists constants $C_j(T )$, $j=0,1,\dots$ such that for all $0\leq t<T $
	\begin{align*}
	\left| \frac{\partial^j u_t}{\partial x^{a_1}\partial x^{a_2}\dots \partial x^{a_j}  } \right|\leq C_j(T).
	\end{align*}
	Lemma \ref{lem:C00} and Lemma \ref{lem:C11} provide such bounds for the case $j=0$ and $j=1$ respectively. From the view of \eqref{hessianu3D} and Lemma \ref{lem:ht}, Lemma \ref{lem:C22} shows the case $j=2$. Furthermore, by differentiating \eqref{hessianu3D}, it suffices to show that for all $j\geq 1$,
	\begin{align}\label{Cjht}
	|\widetilde{D}^j \tilde{h}|_{\tilde{\gamma}}\leq C_{j+2}(T).  
	\end{align}
	Here $\widetilde{D}$ is the Levi-Civita connection of $\tilde{\gamma}$. Compute
	\begin{align*}
	\frac{\partial}{\partial t} \tilde{\gamma}(\widetilde{D}^j \tilde{h},\widetilde{D}^j \tilde{h})=\underbrace{\left(\frac{\partial}{\partial t} \tilde{\gamma}\right)(\widetilde{D}^j \tilde{h},\widetilde{D}^j \tilde{h})}_{\textsc{I}}+\underbrace{ \tilde{\gamma}\left( \left[\frac{\partial}{\partial t},\widetilde{D}^j\right] \tilde{h},\widetilde{D}^j \tilde{h}\right)}_{\textsc{II}}+\underbrace{ \tilde{\gamma} \left(   \widetilde{D}^j \frac{\partial}{\partial t} \tilde{h},\widetilde{D}^j \tilde{h}\right)}_{\textsc{III}}.
	\end{align*}
	Because $\frac{\partial}{\partial t}\tilde{\gamma}_{ab}=2\tilde{h}_{ab}$, $$\textsc{I}=\widetilde{D}^j  \tilde{h}\ast \widetilde{D}^j  \tilde{h}\ast \tilde{h}.$$
For any tensor $N$, we have the commutation relation $[\frac{\partial}{\partial t},\widetilde{D}^j] N=\sum_{i=0}^{j-1} \widetilde{D}^{j-i} \tilde{h}\ast \widetilde{D}^i N$. Hence
	\begin{align*}
	\textsc{II}=\sum_{i=0}^{j}\widetilde{D}^j  \tilde{h}\ast \widetilde{D}^{j-i}  \tilde{h}\ast \widetilde{D}^{i}\tilde{h}.
	\end{align*}
	Finally, from \eqref{equ:ht}, 
	\begin{align*}
	\textsc{III}=\sum_{i=0}^{j}\widetilde{D}^j  \tilde{h}\ast \widetilde{D}^{j-i}  \tilde{h}\ast \widetilde{D}^{i}\tilde{h}+\widetilde{D}^j K\ast \widetilde{D}^j\tilde{h} .
	\end{align*}
	Here
	\begin{align*}
	K_{ab}=\left(\frac{\partial}{\partial x^a} \right)^\alpha \tilde{\nu}^\beta \left(\frac{\partial}{\partial x^b} \right)^\lambda \tilde{\nu}^\mu  \tilde{R}_{\alpha \beta\lambda\mu}.
	\end{align*}
	Let 
	$$\omega_j(t)=\max_{\Sigma_t}|\widetilde{D}^j \tilde{h}|^2_{\tilde{\gamma}}.$$ 
	We use induction and assume $\omega_i(t)\leq C_i(T)$ for $0\leq i\leq j-1$ and $t\in [0,T) $. By Lemma \ref{lem:nablaR}, $|\widetilde{D}^i K|_{\tilde{\gamma}}$ is bounded for $0\leq i\leq j$. We then deduce
	\begin{align*}
	\frac{d}{dt}\omega_j\leq C\omega_j+C.
	\end{align*}
	From the ODE comparison, $\omega_j$ remains bounded in $t\in [0,T)$. Thus \eqref{Cjht} follows.
\end{proof}

\section{Four-dimensional with symmetry in $\xi$}\label{Sec:4D}
In this section, we fix $n=4$ and prove Theorem \ref{thm:4d}. Recall that $F:[0,T_0)\times T^{n-1}\to M$ solves
\begin{equation}\label{normalflow_4D}
\frac{\partial F}{\partial t} =\left(\phi\frac{d\phi}{ds} \right)^{-1} \nu.
\end{equation} 
Here we take $T_0\in (0,\infty]$ to be the largest number such that \eqref{normalflow_4D} has a smooth graphical solution in $t\in[0,T_0).$\\

Equation \eqref{normalflow_4D} can be viewed as the unit normal flow with respect to the conformal metric $\bar{g}=\left(\phi\frac{d \phi}{ds}\right)^2 g $. Recall that $\bar{h}_{ab}$ is the the second fundamental form with respect to $\bar{g}$. Our main assumption is  
\begin{equation}\label{assumption_4D}
\begin{split}
\bar{h}_{ab}\ \textup{is non-negative definite for } \Sigma_0\ \textup{and}\ \min_{\Sigma_0}\phi^4\geq 1+\frac{2}{\sqrt{3}}.
\end{split} 
\end{equation}
We list two lemmas which allow us to prove Theorem \ref{thm:4d}. 

\begin{lemma}\label{lem:limitQ_4D}
	Suppose $T_0=\infty$. Then
	\begin{align}\label{limitQ_4D}
	\limsup_{t\to\infty}Q(\Sigma_t)\leq -2m.
	\end{align} 
\end{lemma}

\begin{lemma}\label{lem:T0_4D}
	Under the assumption \eqref{assumption_4D}	, we have $T_0=\infty$.
\end{lemma}

\begin{proof}[Proof of Theorem \ref{thm:4d}]
	Let $\Sigma_t$ be the solution of \eqref{normalflow_4D} starting from $\Sigma_0=\Sigma$. Combining Lemmas \ref{lem:monotone}, \ref{lem:T0_4D} and  \ref{lem:limitQ_4D},
	\begin{align*}
	Q(\Sigma)\leq \limsup_{t\to\infty}Q(\Sigma_t)\leq -2m.
	\end{align*}
		Furthermore, the equality implies $Q(\Sigma_t)$ is a constant along the flow. And Lemma \ref{lem:monotone} implies $\Sigma$ is a coordinate torus. 
\end{proof}
We prove Lemmas \ref{lem:limitQ_4D} and Lemma \ref{lem:T0_4D} in the rest of this section. We again adapt the convention that $C$ denotes a large constant depending on $\Sigma_0$. The value of $C$ may change from line to line.\\

Denote by $v_t(\xi,\theta)$ the height function of $\Sigma_t$ and $u_t =q(v_t)$. Then $v_t$ solves the equation
\begin{equation}
\frac{\partial v_t}{\partial t}=\rho \left( \phi(v_t)\frac{d\phi}{ds}(v_t)\right) ^{-1}.
\end{equation} 
Here $\rho$ is the slope of $\Sigma_t$ defined in \eqref{v}. Because $\Sigma_t$ are symmetric along the $\xi$ direction, the slope $\rho$ is given by
\begin{align*}
\rho^2=1+\delta^{ij}  \frac{\partial u_t}{\partial\theta^i}\frac{\partial u_t}{\partial\theta^j}.
\end{align*} 
We start with the $C^0$ and $C^1$ estimates for $v_t$ and $u_t$. 
\begin{lemma}\label{lem:C00_4D}
	There exists a constant $C $ depending on $\Sigma_0$ such that for all $t\in [0,T_0)$,
	\begin{equation}\label{C00_4D}
	\begin{split}
	|\phi^2(v_t)-2t|\leq C.
	\end{split}
	\end{equation} 
\end{lemma}
\begin{proof}
	At the maximum point of $v_t$, we have $\rho=1$. Together with \eqref{equ:phi}, 
	$$\frac{d}{dt}\max_{T^2} \phi^2(v_t)\leq  2. $$
	Similarly,
	$$\frac{d}{dt}\min_{T^2} \phi^2(v_t)\geq  2. $$
	Thus the assertion follows by taking $C =\max_{T^2}\phi^2(v_0)$.
\end{proof}

\begin{lemma}\label{lem:C11_4D}
	There exists constant $C $ depending on $\Sigma_0$ such that for $t\in [0,T_0)$,
	\begin{equation}\label{C11_4D}
\max_{\Sigma_t}(	\rho^2-1) \leq  (C^{-1}t+1)^{-3} \max_{\Sigma_0}(\rho^2-1). 
	\end{equation}
\end{lemma}

\begin{proof}
	
	From
	\begin{align*}
	\frac{\partial^2 u_t}{\partial x^a\partial t}=&\phi^{-2}\left(\frac{d\phi}{ds} \right)^{-1} \frac{\partial \rho}{\partial x^a}+\rho\phi^{-2} (-3+\phi^{-4})(1-\phi^{-4})^{-1} \frac{\partial u_t}{\partial x^a},  
	\end{align*}   
	we deduce
	\begin{align*}
	\frac{\partial  }{\partial t}(\rho^2-1)  =&2\phi^{-2}\left(\frac{d\phi}{ds} \right)^{-1} \delta^{ij}  \frac{\partial u_t}{\partial\theta^i}\frac{\partial \rho}{\partial\theta^j}+\rho\phi^{-2} (-3+\phi^{-4})(1-\phi^{-4})^{-1} (\rho^2-1).
	\end{align*}
	At the maximum point of $\rho^2-1$, $d\rho	=0$ and
	\begin{align*}
	\frac{\partial}{\partial t}(\rho^2  -1) \leq  &\rho\phi^{-2} (-3+\phi^{-4})(1-\phi^{-4})^{-1} (\rho^2-1).
	\end{align*}
	Using \eqref{C00_4D}, there exists a constant $C $ such that for all $t\in [0,T_0)$,
	\begin{align*}
	\rho\phi^{-2} (-3+\phi^{-4})(1-\phi^{-4})^{-1} \geq \phi^{-2} (-3+\phi^{-4})(1-\phi^{-4})^{-1} \geq -3(t+C )^{-1}.
	\end{align*}
 Hence
	\begin{align*}
	\frac{d}{dt}\max_{\Sigma_t} (\rho^2-1)\leq -3(t+C )^{-1}\max_{\Sigma_t} (\rho^2-1).
	\end{align*}
	By the ODE comparison, 
	\begin{align*}
	\max_{\Sigma_t}(\rho^2-1) \leq  (C^{-1}t+1)^{-3} \max_{\Sigma_0}(\rho^2-1).
	\end{align*}
	The proof is finished.
\end{proof}

Recall that $g'=\phi^{-2}g$ is a conformal metric which asymptotic to the flat metric on $\mathbb{R}\times T^3$. Let $\gamma'_t$ be the induced metric of $g'$ on $\Sigma_t$ and $D'$ be the Levi-Civita connection of $\gamma'_t$. Lemmas \ref{lem:C00_4D} and \ref{lem:C11_4D} imply, provided $T_0=\infty$, $\gamma'_t$ converges the the flat metric $d\xi^2+ d\theta^2$ in $C^0$. Next, we use the second fundamental form $h'_{ab}$ to bound the hessian of $u_t$ from below. 

\begin{lemma}
	Suppose $T_0=\infty.$ Then there exists a constant $C$ depending on $\Sigma_0$ such that  for $t\geq C+1$,
	\begin{align*}
	D'_b(D')^a u_t \geq -  \left( C(t-C)^{-3/2}\log t\right) \delta_b^a. 
	\end{align*}
\end{lemma}

\begin{proof}
The flow \eqref{normalflow_4D} can be rewritten as
	\begin{align*}
	\frac{\partial F }{\partial t}= \left(\phi^{2} \frac{d\phi}{ds}\right)^{-1} \nu'.
	\end{align*} 
	Here $\nu'=\phi\nu$ is the unit normal vector with respect to $g'$. Let $G= \phi^{2} \frac{d\phi}{ds} $. Then the evolution equation of $(h')_b^a$ is given by
	\begin{align*}
	\frac{\partial}{\partial t}(h')_b^a=-G^{-1}(h')_b^c(h')_c^a-D'_b(D')^a G^{-1}-G^{-1} (dx^a)_\alpha (\nu')^\beta \left(\frac{\partial}{\partial x^b} \right)^\lambda (\nu')^\mu   ({R}')^\alpha_{\ \beta\lambda\mu}.
	\end{align*}
	Here $(R')_{\alpha\beta\lambda\mu}$ is the Riemannian curvature tensor of $g'$. In the estimates below, the norms are computed with respect to $\gamma'_t$. From \eqref{curvature-g'} and \eqref{equ:phi}, the only non-zero component of $R'$ is
	\begin{align*}
	R'_{q\xi q\xi}=6\phi^{-2}(1-\phi^{-4}).
	\end{align*} 
	Together with \eqref{C00_4D}, for $t$ large enough,
	\begin{align*}
	\left| -G^{-1} (dx^a)_\alpha (\nu')^\beta \left(\frac{\partial}{\partial x^b} \right)^\lambda (\nu')^\mu   ({R}')^\alpha_{\ \beta\lambda\mu} \right|\leq C(t-C)^{-5/2}.
	\end{align*} 
	Using \eqref{equ:q},
	\begin{align*}
	-D'_b(D')^a G^{-1} =&G^{-2}\frac{d G}{dq}  D'_b(D')^a {u}_t+\frac{d}{dq}\left( G^{-2}\frac{dG}{dq}\right) D'_b {u}_t (D')^a{u}_t.
	\end{align*} 
 The relation between $h'_{ab}$ and $D'_bD'_a u_t$ is given in Corollary \ref{cor:h-hessian} with $\Psi=\phi^{-1}\frac{d\phi}{ds}$. Using \eqref{equ:phi}, we derive
	\begin{align}\label{hessianu4D}
	D'_bD'_a u_t=-\rho^{-1}h'_{ab}+2\phi^{-3}(1-\phi^{-4})^{1/2}D'_b \xi D'_a\xi.
	\end{align}
	Together with
	\begin{align*}
	G^{-2}\frac{d G}{dq}=&\phi^{-2}(3-\phi^{-4})(1-\phi^{-4})^{-1},\\
	\frac{d}{dq}\left( G^{-2}\frac{dG}{dq}\right)=&-2\phi^{-1}(3+\phi^{-8})(1-\phi^{-4})^{-3/2},
	\end{align*}
	we derive
	\begin{align*}
	-D'_b(D')^aG^{-1}=&-\rho^{-1}\phi^{-2}(3-\phi^{-4})(1-\phi^{-4})^{-1}  (h')_b^a-2\phi^{-1}(3+\phi^{-8})(1-\phi^{-4})^{-3/2}  D'_b {u}_t (D')^a{u}_t\\
	&  +2\phi^{-5}(3-\phi^{-4})(1-\phi^{-4})^{-1/2}   D'_b \xi  (D')^a \xi .
	\end{align*}
	From \eqref{C00} and \eqref{C11}, for $t$ large enough,
	\begin{align*}
	&  |2\phi^{-5}(3-\phi^{-4})(1-\phi^{-4})^{-1/2}   D'_b\xi (D')^a  \xi |\leq  C(t-C)^{-5/2}. 
	\end{align*}
	Putting the above together, we derive
	\begin{align*}
	\bigg|\frac{\partial}{\partial t}(h')_b^a+\rho^{-1}\phi^{-2}&(3-\phi^{-4})(1-\phi^{-4})^{-1}  (h')_b^a+G^{-1}(h')_b^c(h')_c^a \\
	 &+2\phi^{-1}(3+\phi^{-8})(1-\phi^{-4})^{-3/2}  D'_b {u}_t (D')^a{u}_t \bigg|\leq C(t-C)^{-5/2}.
	\end{align*}
	Let $w_t$ be the maximum eigenvalue of $(h')_a^b$. From \eqref{C00} and \eqref{C11}, for $t$ large enough,
	\begin{align*}
		&\rho^{-1} \phi^{-2}(3-\phi^{-4})(1-\phi^{-4}) \geq  \frac{3}{2} (t+C)^{-1}. 
	\end{align*} We derive
	\begin{align*}
	\frac{d w_t}{d t}\leq -\frac{3}{2}(t+C)^{-1} w_t+C(t-C)^{-5/2}.
	\end{align*}
	From the ODE comparison, $w_t\leq C(t-C)^{-3/2}\log t $ for $t$ large enough. Hence the assertion follows \eqref{hessianu4D}.
\end{proof}

\begin{proof}[Proof of Lemma \ref{lem:limitQ_4D}]
	From \eqref{limit},
	\begin{align*}
	Q(\Sigma_t)+2m=&\int_{T^3 }  \rho^2\phi \frac{d\phi}{ds}\Delta' u_t- 2\delta^{ij}  \frac{\partial u_t}{\partial\theta^i}\frac{\partial u_t}{\partial\theta^j}  \, d\xi\wedge d\theta^3\wedge d\theta^4. 
	\end{align*}
	From Lemma \ref{lem:C11_4D}, $\delta^{ij}  \frac{\partial u_t}{\partial\theta^i}\frac{\partial u_t}{\partial\theta^j} \leq C(t-C)^{-3}$. It remains to prove
	\begin{align*}
	\limsup_{t\to\infty} \int_{T^3}  \rho ^2\phi \frac{d\phi}{ds}\Delta' u_t \, d\xi\wedge d\theta^3\wedge d\theta^4 \leq 0. 
	\end{align*}
	Recall that $\gamma'_t$ is the induced metric of $\Sigma$ from $g'$. Because $$d\textup{vol}_{\gamma'_t}=\rho \phi^{-1}\frac{d\phi}{ds}d\xi\wedge d\theta^3\wedge d\theta^4,$$
	\begin{align*}
	&\int_{ {T}^{3}}  \rho^2\phi \frac{d\phi}{ds}\Delta' u_t d\xi\wedge d\theta^3\wedge d\theta^4 =  \int_{\Sigma_t} \rho \phi^2  \Delta' u_t \, d\textup{vol}_{\gamma'_t}=\underbrace{ \int_{\Sigma_t}  (\rho-1) \phi^2  \Delta' u_t \, d\textup{vol}_{\gamma'_t} }_{\textsc{I}}+\underbrace{ \int_{\Sigma_t}   \phi^2  \Delta' u_t \, d\textup{vol}_{\gamma'_t} }_{\textsc{II}}.
	\end{align*}
	From Lemmas \ref{lem:C00_4D}, \ref{lem:C11_4D} and \ref{lem:C22_4D}, $\phi^2\leq Ct$, $|\rho-1|\leq Ct^{-3}$ and
	\begin{align*}
	\int_{\Sigma_t}   |\Delta' u_t| \, d\textup{vol}_{\gamma'}=2\int_{\Sigma_t}   (\Delta' u_t)_- \, d\textup{vol}_{\gamma'}\leq C t^{-3/2}\log t.
	\end{align*}  
	Hence the first term $\textsc{I}$ goes to zero. Through integration by parts,
	\begin{align*}
	\textsc{II}=-\int_{\Sigma_t} 2\phi \frac{d\phi}{ds} |D' u_t|^2_{\gamma'_t} \, d\textup{vol}_{\gamma'}\leq 0.
	\end{align*}
	Thus \eqref{limitQ_4D} follows.
\end{proof} 

Lastly, we prove Lemma \ref{lem:T0_4D}. Recall that $\bar{g}=\left(\phi\frac{d\phi}{ds} \right)^2 g$ and that $\bar{h}_b^a$ is the Weingarten tensor with respect to $\bar{g}$. The flow \eqref{normalflow_4D} is the unit normal flow with respect to $\bar{g}$. The evolution equation of $\bar{h}_b^a$ is given by
\begin{equation}\label{equ:hb}
\frac{\partial}{\partial t}\bar{h}_b^a=-\bar{h}_b^c\bar{h}_c^a-(dx^a)_\alpha \bar{\nu}^\beta \left(\frac{\partial}{\partial x^b} \right)^\lambda \bar{\nu}^\mu  \overline{R}^\alpha_{\ \beta\lambda\mu}.
\end{equation}

\begin{lemma} \label{lem:C22_4D}
	Under the assumption \eqref{assumption_4D},  there exists a continuous function $C (t)$ defined on $[0,\infty)$ such that for all $t\in [0,T_0)$,
	\begin{align*}
	|\bar{h}_b^a|_{\bar{\gamma}_t}\leq |\bar{h}_b^a|_{\bar{\gamma}_0}+C(t) 
	\end{align*}
\end{lemma}
\begin{proof}
	Using \eqref{curvature-gt} with $\Psi=\phi^{-1}\frac{d\phi}{ds}$ and $\psi=2\log\phi+\log\frac{d\phi}{ds}$, we derive that non-zero components of $\overline{R}_{\alpha\beta\lambda\mu}$ are 
	\begin{align*}
	\overline{R}_{q\xi q\xi}=&-3\phi^{8}(1-\phi^{-4})^3 ,\\
	\overline{R}_{qi  qi}=&-\phi^8(3-6\phi^{-4}-\phi^{-8}) ,\\
	\overline{R}_{\xi i \xi i}=&-\phi^{8}(9-\phi^{-8})(1-\phi^{-4}) ,\\
	\overline{R}_{ijij}=&-\phi^8(1-\phi^{-4})^2,\ i\neq j.
	\end{align*}
Besides $\overline{R}_{qi  qi}$, other sectional curvatures are apparently non-positive. Actually, the polynomial $3x^2-6x-1$ has roots $x=1\pm\frac{2}{\sqrt{3}}$. Under the assumption \ref{assumption_4D}, we have $\phi^4\geq 1+\frac{2}{\sqrt{3}}$ along the flow. Hence the term
	$$-(dx^a)_\alpha \bar{\nu}^\beta \left(\frac{\partial}{\partial x^b} \right)^\lambda \bar{\nu}^\mu  \overline{R}^\alpha_{\ \beta\lambda\mu}$$
	is non-negative definite and $\bar{h}_b^a$ also remains non-negative definite . On the other hand, from \eqref{C00_4D}, for any $t\in [0,T_0)$,
	\begin{align*}
	\left| (dx^a)_\alpha \bar{\nu}^\beta \left(\frac{\partial}{\partial x^b} \right)^\lambda \bar{\nu}^\mu  \overline{R}^\alpha_{\ \beta\lambda\mu} \right|\leq C(t).
	\end{align*}
	for some function $C(t)$. Then the assertion follows by the ODE comparison. 
\end{proof}  
With Lemma \ref{lem:C22_4D}, the proof of Lemma \ref{lem:T0_4D} is similar to the one for Lemma \ref{lem:T0_3D}.
\begin{appendix}
\section{Curvature and Second Fundamental Form}\label{sec:app1}
	In this appendix, we record formulas of curvatures and second fundamental forms. Let $\Psi(q)$ be a smooth and positive function. Consider the metric 
	\begin{align}\label{def:g'}
	g'=dq^2+\Psi(q)^2d\xi^2+\sum_{i=3}^n (d\theta^i)^2. 
	\end{align}
	The non-zero components of the Christoffel symbols are
	\begin{align}\label{Christ-gp}
	(\Gamma')_{q\xi}^\xi=\Psi^{-1}\frac{d\Psi}{d q},\ (\Gamma')_{\xi\xi}^q=-\Psi\frac{d\Psi}{d q}.
	\end{align}
	The only non-zero component of the curvature is
	\begin{align}\label{curvature-g'}
	R'_{q\xi q\xi}=-\Psi\frac{d^2 \Psi}{dq^2}.
	\end{align}
	
	We now consider a hypersurface given by a graph over $T^{n-1}$. Let $u(\xi,\theta^i)$ be a smooth function and consider the map $F(\xi,\theta^i)=(u(\xi,\theta^i),\xi,\theta^i)$. Denote the image of $F$ by $\Sigma$. 
	\begin{align*}
	\frac{\partial F}{\partial \xi}=\frac{\partial }{\partial \xi}+\frac{\partial u}{\partial \xi}\frac{\partial}{\partial q},\ \frac{\partial F}{\partial \theta^i}=\frac{\partial }{\partial \theta^i}+\frac{\partial u}{\partial \theta^i}\frac{\partial}{\partial q}.
	\end{align*}  
	Let $\{x^a\}=\{\xi,\theta^i\}$ and define a metric on $T^{n-1}$ as
	\begin{equation}
	\hat{\gamma}_{ab}dx^a dx^b= \Psi(u)^2 d\xi^2+\delta_{ij}d\theta^i d\theta^j. 
	\end{equation}
	The induced metric $\gamma'_{ab}$ on $\Sigma$ is given by
	\begin{align}
	\gamma'_{ab}=\hat{\gamma}_{ab}+\frac{\partial u}{\partial x^a}\frac{\partial u}{\partial x^b}. 
	\end{align}
	The normal vector of $\Sigma$ is given by
	\begin{align}\label{normal-gp}
	\nu'=\rho^{-1}\left(\frac{\partial}{\partial q}-\frac{1}{ \Psi(u)^{2}}\frac{\partial u}{\partial\xi}\frac{\partial }{\partial\xi}-\delta^{ij} \frac{\partial u}{\partial\theta^i}\frac{\partial  }{\partial\theta^j} \right), 
	\end{align}
	where $\rho$ is the slope given by
	\begin{align}
	\rho^2= 1+\Psi(u)^{-2}\left(\frac{\partial u}{\partial\xi} \right)^2+\delta^{ij}\frac{\partial u}{\partial\theta^i}\frac{\partial u}{\partial\theta^j}. 
	\end{align}
	\begin{lemma}\label{lemma:A1}
		The second fundamental form of $\Sigma$ is given by
		\begin{align*}
		h'_{\xi\xi}=&-\rho^{-1}\frac{\partial^2 u}{\partial \xi^2}+2\rho^{-1}\Psi^{-1}\frac{d\Psi}{dq}\left(\frac{\partial u}{\partial \xi} \right)^2+\rho^{-1}\Psi\frac{d\Psi}{d q},\\
		h'_{ij}=&-\rho^{-1}\frac{\partial^2 u}{\partial \theta^i\partial\theta^j },\\
		h'_{\xi i}=&-\rho^{-1}\frac{\partial^2 u}{\partial \xi\partial\theta^i }+\rho^{-1}\Psi^{-1}\frac{d \Psi}{dq}\frac{\partial u}{\partial\xi}\frac{\partial u}{\partial\theta^i}.
		\end{align*}
	\end{lemma} 
	\begin{proof}
		Let $\nabla'$ be the Levi-Civita connection of $g'$. From \eqref{Christ-gp}, 
		\begin{align*}
		\nabla'_{\frac{\partial F}{\partial\xi}}\frac{\partial F}{\partial\xi}=&\left( \frac{\partial^2 u}{\partial \xi^2}-\Psi\frac{d\Psi}{dq} \right)\frac{\partial}{\partial q}+2\Psi^{-1}\frac{d\Psi}{dq}\frac{\partial u}{\partial\xi}\frac{\partial}{\partial \xi},\\
		\nabla'_{\frac{\partial F}{\partial\theta^i}}\frac{\partial F}{\partial\theta^j} =&\frac{\partial^2 u}{\partial\theta^i\theta^j}  \frac{\partial}{\partial q}, \\
		\nabla'_{\frac{\partial F}{\partial\theta^i}}\frac{\partial F}{\partial\xi}= & \frac{\partial^2 u}{\partial \xi\partial\theta^i } \frac{\partial}{\partial q}+ \Psi^{-1}\frac{d\Psi}{dq}\frac{\partial u}{\partial\theta^i}\frac{\partial}{\partial \xi}.
		\end{align*}
		Taking inner product with $\nu'$ given in \eqref{normal-gp}, the assertion follows.
	\end{proof}
	We now view $u$ as a function defined on $\Sigma$ and compute its hessian. Let $D'$ be the Levi-Civita connection of $\gamma'$.
	\begin{lemma}\label{lemma:A2}
		\begin{align*}
		D_\xi'D_\xi' u=&\rho^{-2}\frac{\partial \tilde{u}}{\partial \xi^2}   - 2\rho^{-2}\Psi^{-1}\frac{d\Psi}{dq} \left( \frac{\partial u}{\partial \xi}\right)^2 -\rho^{-2}\Psi\frac{d\Psi}{dq}+  \Psi\frac{d\Psi}{dq},\\
		D_i'D_j' u=&\rho^{-2}\frac{\partial^2 u}{\partial\theta^i\partial\theta^j},\\
		D_\xi'D_i' u=&\rho^{-2}\frac{\partial^2 u}{\partial\xi\partial\theta^i}-\rho^{-2}\Psi^{-1}\frac{d\Psi}{dq} \frac{\partial u}{\partial\xi}\frac{\partial u}{\partial\theta^i}.
		\end{align*}
	\end{lemma}
	\begin{proof}
		Let $\hat{\Gamma}_{ab}^c$ be the Christoffel symbols of $\hat{\gamma}_{ab}$. The non-zero components of $\hat{\Gamma}_{ab}^c$ are
		\begin{align*}
		\hat{\Gamma}_{\xi i}^\xi=\Psi^{-1}\frac{d\Psi}{d q}\frac{\partial u}{\partial\theta^i},\ \hat{\Gamma}_{\xi \xi}^i=-\Psi \frac{d\Psi}{d q}\frac{\partial u}{\partial\theta^i},\ \hat{\Gamma}_{\xi \xi}^\xi=\Psi^{-1}\frac{d\Psi}{d q}\frac{\partial u}{\partial\xi}.
		\end{align*}
		For any smooth function $\tilde{u}$, 
		\begin{align*}
		\hat{D}_\xi\hat{D}_\xi  \tilde{u}=& \frac{\partial \tilde{u}}{\partial \xi^2}+ \Psi\frac{d\Psi}{dq}\delta^{ij} \frac{\partial u}{\partial\theta^i}\frac{\partial \tilde{u}}{\partial\theta^j} - \Psi^{-1}\frac{d\Psi}{dq} \frac{\partial u}{\partial \xi}\frac{\partial \tilde{u}}{\partial \xi}, \\
		\hat{D}_i \hat{D}_j  \tilde{u}=& \frac{\partial^2 \tilde{u}}{\partial\theta^i\partial\theta^j},\\
		\hat{D}_\xi \hat{D}_i  \tilde{u}=& \frac{\partial^2 \tilde{u}}{\partial\xi\partial\theta^i}- \Psi^{-1}\frac{d\Psi}{dq} \frac{\partial u}{\partial\theta^i}\frac{\partial \tilde{u}}{\partial\xi}.
		\end{align*}
		Here $\hat{D}$ is the Levi-Civita connection for $\hat{\gamma}$. Because $\gamma'_{ab}=\hat{\gamma}_{ab}+\frac{\partial u}{\partial x^a}\frac{\partial u}{\partial x^b}$, we have
		\begin{align*}
		D_a'D_b' u=\rho^{-2}\hat{D}_a\hat{D}_b  u.
		\end{align*}
		Combining the above, the assertion for $D'_iD'_j u$ and $D'_\xi D'_i u$ follows. For $D'_\xi D'_\xi u$, we use $$\delta^{ij}\frac{\partial u}{\partial\theta^i}\frac{\partial u}{\partial\theta^j}=-\Psi^{-2}\left(\frac{\partial u}{\partial \xi} \right)^2-1+\rho^2.$$
		Then
		\begin{align*}
		D'_\xi D'_\xi u=&\rho^{-2}\hat{D}_\xi\hat{D}_\xi  {u}=\rho^{-2}\left(\frac{\partial \tilde{u}}{\partial \xi^2}+ \Psi\frac{d\Psi}{dq}\delta^{ij} \frac{\partial u}{\partial\theta^i}\frac{\partial  {u}}{\partial\theta^j} - \Psi^{-1}\frac{d\Psi}{dq}\left( \frac{\partial u}{\partial \xi}\right)^2 \right)\\
		=&\rho^{-2}\left(\frac{\partial \tilde{u}}{\partial \xi^2}   - 2\Psi^{-1}\frac{d\Psi}{dq} \left( \frac{\partial u}{\partial \xi}\right)^2 -\Psi\frac{d\Psi}{dq}+\rho^2 \Psi\frac{d\Psi}{dq}\right)\\
		= &\rho^{-2}\frac{\partial \tilde{u}}{\partial \xi^2}   - 2\rho^{-2}\Psi^{-1}\frac{d\Psi}{dq} \left( \frac{\partial u}{\partial \xi}\right)^2 -\rho^{-2}\Psi\frac{d\Psi}{dq}+  \Psi\frac{d\Psi}{dq}.
		\end{align*}
	\end{proof}
	From Lemmas \ref{lemma:A1} and \ref{lemma:A2}, we obtain the following relation between $h'_{ab}$ and $D'_bD'_bu$.
	\begin{cor}\label{cor:h-hessian}
		\begin{equation}\label{h'-hessian}
		\begin{split}
		h'_{ab} =-\rho D'_a D'_b u +\rho\Psi\frac{d \Psi}{dq} D'_a\xi D'_b\xi .
		\end{split}
		\end{equation}
	\end{cor}
	Next, we record the formulas under a conformal transformation. Let $\psi(q)$ be a smooth function of $q$. Define  $$ \check{g}=e^{2\psi}g'.$$ Let $\check{R}_{\alpha\beta\lambda\mu}$ be the Riemannian curvature of $\check{g}$.
	\begin{lemma}
		The non-zero components of $\check{R}_{\alpha\beta\lambda\mu}$ are
		\begin{equation}\label{curvature-gt}
		\begin{split}
		\check{R}_{q\xi q\xi}=&-e^{2\psi}\left( \Psi\frac{d^2\Psi}{dq^2}+\Psi^2\frac{d^2\psi}{dq^2}+\Psi\frac{d\Psi}{dq}\frac{d\psi}{dq} \right),\\
		\check{R}_{qi qi}=&  -e^{2\psi}\frac{d^2\psi}{dq^2}  ,\\
		\check{R}_{\xi i\xi  i}=&  -e^{2\psi}\Psi\frac{d\Psi}{dq} \frac{d \psi}{dq }-e^{2\psi}\Psi^2\left(\frac{d\psi}{dq} \right)^2,\\
		\check{R}_{ ij i j}=&  -e^{2\psi} \left(\frac{d\psi}{dq} \right)^2,\ i\neq j.
		\end{split}
		\end{equation}
		\begin{proof}
			As $\check{g}=e^{2\psi}g'$, the formula of $\check{R}_{\alpha\beta\lambda\mu}$ is given by \cite[page 58]{Besse} 
			\begin{align*}
			\check{R}_{\alpha\beta\lambda\mu}=e^{2\psi}{R}'_{\alpha\beta\lambda\mu}-e^{2\psi}(g'_{\alpha\lambda}T_{\beta\mu}+g'_{\beta\mu}T_{\alpha\lambda}-g'_{\alpha\mu}T_{\beta\lambda}-g'_{\beta\lambda}T_{\alpha \mu}),
			\end{align*}
			with $T_{\alpha \beta}=\nabla'_\alpha\nabla'_\beta\psi-\nabla'_\alpha\psi\nabla'_\beta\psi+2^{-1}|\nabla'\psi|^2_{g'}g'_{\alpha\beta}$. From \eqref{Christ-gp}, non-zero components of $T_{\alpha\beta}$ are
			\begin{align*}
			T_{qq}=&\frac{d^2\psi}{dq^2}-2^{-1}\left(\frac{d\psi}{dq} \right)^2,\\
			T_{\xi\xi}=&\Psi\frac{ d\Psi}{dq}\frac{d\psi}{dq}+2^{-1}\Psi^2 \left(\frac{d\psi}{dq} \right)^2,\\
			T_{ii}= &2^{-1}\left(\frac{d\psi}{dq} \right)^2.
			\end{align*}
			Combining with \eqref{curvature-g'}, the assertion follows a direct computation.
		\end{proof}
	\end{lemma} 
	We turn to the second fundamental form of $\Sigma$ with respect to $\check{g}$. The normal vector and the induced metric are given by $\check{\nu}=e^{-\psi}\nu'$ and $\check{\gamma}=e^{2\psi}\gamma'$ respectively. Let $\check{h}_{ab}$ be the second fundamental form with respect to $\check{g}$.
	\begin{lemma}\label{lem:ht}
		\begin{align}\label{hcheck}
		\check{h}_{ab}=e^\psi  \left(h'_{ab}+ \rho^{-1}\frac{d\psi}{dq} \gamma'_{ab}\right).
		\end{align}
	\end{lemma}
	\begin{proof}
		Let $\check{\nabla}$ be the Levi-Civita connection of $\check{g}$. From
		\begin{align*}
		\check{\nabla}_X Y= \nabla'_X Y+g'( \nabla'\psi,X)Y+g'( \nabla'\psi,Y)X-g'( X,Y)\nabla'\psi,
		\end{align*}
		\begin{align*}
		\check{h}_{ab}=-\check{g}\left( \check{\nabla}_{\frac{\partial F}{\partial x^a}}  \frac{\partial F}{\partial x^b},\check{\nu} \right)=-e^\psi g'\left (  \nabla'_{\frac{\partial F}{\partial x^a}}  \frac{\partial F}{\partial x^b}-\gamma'_{ab}\frac{d\psi}{dq }\frac{\partial}{\partial q},\nu \right)=e^\psi  \left(h'_{ab}+\rho^{-1}\frac{d\psi}{dq} \gamma'_{ab}\right).
		\end{align*}
	\end{proof}
	Using \eqref{h'-hessian}, \eqref{hcheck}, we have
	\begin{align*}
			\check{h}_{ab}=e^\psi  \left(-\rho D'_aD'_b u+\rho\Psi\frac{d\Psi}{dq}D'_a\xi D'_b\xi + \rho^{-1}\frac{d\psi}{dq} \gamma'_{ab}\right).
\end{align*}	
To get the mean curvature $\check{H}$, we need to calculate $|D'\xi|^2_{\gamma'}$. Recall that $$\gamma'_{ab}=\hat{\gamma}_{ab}+\frac{\partial u}{\partial x^a}\frac{\partial u}{\partial x^b}.$$
We deduce
\begin{align*}
(\gamma')^{ab}=\hat{\gamma}^{ab}-\rho^{-2}\hat{\gamma}^{ac}\hat{\gamma}^{bd} \frac{\partial u}{\partial x^c}\frac{\partial u}{\partial x^d}. 
\end{align*} Hence
\begin{align*}
|D'\xi|^2_{\gamma'}=&(\gamma')^{\xi\xi}=\Psi^{-2}-\rho^{-2}\Psi^{-4}\left(\frac{\partial u}{\partial\xi} \right)^2\\
=&\rho^{-2}\Psi^{-2}\left( 1+\delta^{ij}\frac{\partial u}{\partial\theta^i}\frac{\partial u}{\partial\theta^j} \right).
\end{align*}   
With $\check{\gamma}^{ab}=e^{-2\psi}(\gamma')^{ab}$, 	we obtain 
	\begin{cor}
		The mean curvature of $\Sigma$ with respect to $\check{g}$ is given by
		\begin{align}\label{H-gt}
		\check{H}=e^{-\psi}\left( -\rho \Delta' u+\rho^{-1}\Psi^{-1}\frac{d\Psi}{ dq}\delta^{ij} \frac{\partial u}{\partial\theta^i}\frac{\partial u}{\partial\theta^j}+(n-1)\rho^{-1}\frac{d\psi}{dq} +\rho^{-1}\Psi^{-1}\frac{d\Psi}{dq}\right).
		\end{align}
		Here $\Delta'$ is the Laplace-Beltrami operator of $\gamma'$.
	\end{cor}
	\section{Projection}
	Let $M$ be a $n$-dimensional manifold with coordinates $\{y^\alpha\}_{\alpha=1}^n$. Let  $g_{\alpha\beta} $ be a Riemannian metric on $M$ and $\Sigma^{n-1}$ be a hypersurface. We use $\{x^a\}_{a=1}^n$ to denote a local coordinates of $\Sigma$. Denote by $\nu^\alpha$ the unit normal vector of $\Sigma$, by $\gamma_{ab}$ the induced metric on $\Sigma$ and by $h_{ab}$ the second fundamental form of $\Sigma$ with respect to $\nu^\alpha$. Define the projection tensor
	\begin{align}
	P^\alpha_\beta=\delta^\alpha_\beta- \nu_\beta \nu^\alpha.
	\end{align}
	Denote by $\nabla$ and $D$ the Levi-Civita connection of $g_{\alpha\beta}$ and $\gamma_{ab}$ respectively.
	
	\begin{lemma}\label{lem:projection}
		Let $\overline{F}_\alpha$ be a one-form and $F_a=P_a^\alpha \overline{F}_\alpha$ be the projection of $\overline{F}_\alpha$. Then
		\begin{align}
		D_a F_b=&P_a^\alpha P_b^\beta \nabla_\alpha \overline{F}_\beta- \overline{F}_{\alpha}\nu^\alpha h_{ab}.\\
		D_a (F_\alpha \nu^\alpha)=& (P^\beta_a \nabla_\beta \overline{F}_\alpha) \nu^\alpha+    F_b h^b_a.
		\end{align}  
	\end{lemma}
	Let 
	\begin{align}\label{def:K}
	K_{ab}=&P^\alpha_a\nu^\beta P^\lambda_b \nu^\mu R_{\alpha\beta\lambda\mu}.
	\end{align}
	
	\begin{lemma}\label{lem:nablaR}
		For any $m\in\mathbb{N}\cup\{0\}$ there exists a degree $m+1$ polynomial $p_m$ such that the following holds. Suppose there exists a constant $C $ such that for all $0\leq j\leq m$,
		\begin{align}\label{assumptionB}
		|\nabla^j Rm|_g\leq C ,\ |D^{j-1}h|_\gamma\leq C .
		\end{align} 
		Then 
		\begin{align*}
		|D^m K |_\gamma \leq p_m(C ).
		\end{align*}
		
	\end{lemma}
	\begin{proof}
		For $m,k\in\mathbb{N}\cup\{0\}$, let $L_{m,k}$ be the collection $4+m-k$ tensors on $\Sigma$ obtained from projections of $\nabla^m Rm\cdot \nu^{\otimes k}$. We adapt the convention that $L_{m,-1}=0$ and $L_{m,k}=0$ if $k\geq m+3$. Note that $K\in L_{0,2}$. From Lemma \ref{lem:projection}, we have
		\begin{align*}
		D L_{m,k}=L_{m+1,k}+L_{m,k-1}\ast h+L_{m,k+1}\ast h. 
		\end{align*}
		Here we write $L \ast h$ for any linear combination formed by contracting $L$ and $h$ by $\gamma$. Through induction, we have for all $m\in \mathbb{N}\cup\{0\}$,  
		\begin{align*}
		D^m L_{0,2}= \sum  L_{j,i}\ast D^{\ell_1}h\ast \cdots \ast D^{\ell_k}h,
		\end{align*}
		where the summation goes over
		\begin{align*}
		0\leq i\leq j+3,\ 0\leq \ell_i,\ \textup{and}\  j+k+\sum_{i=1}^k \ell_i=m.
		\end{align*}
		Under the assumption \eqref{assumptionB}, each term $L_{j,i}$ and $D^{\ell_i}$ above are bounded by $C$. Each individual summand $L_{j,i}\ast D^{\ell_1}h\ast \cdots \ast D^{\ell_k}h$ is bounded by $C^{k+1}$. Then the assertion then follows.
	\end{proof}
	
\end{appendix}


\begin{thebibliography}{99}
\bibitem{BCHMN} H. Barzegar, P. T. Chruściel, M. Hörzinger, M. Maliborski, and L. Nguyen, \textit{On the energy of the Horowitz-Myers metrics}, Phys. Rev. D., \textbf{101} (2020), 024007.

\bibitem{Besse}A. L. Besse, \textit{Einstein manifolds}, Classics in Mathematics. Springer-Verlag,
Berlin, 1987.
 
\bibitem{Brendle} S. Brendle, \textit{Constant mean curvature surfaces in warped product manifolds}, Publ. Math. IHES, \textbf{117} (2013), 247–269.
\bibitem{BrendleHuanWang} S. Brendle, P. K. Hung, and M. T. Wang, \textit{A Minkowski Inequality for Hypersurfaces in the Anti-de Sitter-Schwarzschild Manifold.}, Comm. Pure Appl. Math., \textbf{69} (2016), no. 1, 124–144.


	\bibitem{CostableMyers} N. R. Constable, and R. C. Myers, \textit{Spin two glueballs, positive energy theorems and the AdS/CFT
	correspondence}, J. High Energy Phys., \textbf{037} (1999).
\bibitem{Guan1} P. Guan, X. N. Ma, N. Trudinger, and X. Zhu, \textit{A form of Alexandrov-Fenchel
	inequality}., Pure Appl. Math. Q., \textbf{6} (2010), 999–1012.
\bibitem{Guan2}  P. Guan, and J. Li, \textit{The quermassintegral inequalities for k-convex starshaped
	domain}., Adv. Math. \textbf{221} (2009), 1725-1732.
\bibitem{HM}G.T. Horowitz and R.C. Myers,  \textit{The AdS/CFT Correspondence and a New Positive Energy Conjecture for General Relativity}, Phys. Rev. D, \textbf{59} (1999), 1–12.
\bibitem{McCormickM} S. McCormick, \textit{On a Minkowski-like inequality for asymptotically flat static manifolds}, Proc. Am. Math. Soc., \textbf{146} (2018), no. 9, 4039–4046.
\bibitem{HuiskenIlmanen} G. Huisken, and T. Ilmanen, \textit{The inverse mean curvature flow and the Riemannian Penrose inequality}, J. Differential
Geom., \textbf{59} (2001), 353-437.





\bibitem{YWei} Yong Wei, \textit{On the Minkowski-type inequality for outward minimizing hypersurfaces in Schwarzschild space}, Calc. Var. Partial Differ. Equ. \textbf{57} (2018), 46.
\bibitem{Ewoolgar} E. Woolgar, \textit{The rigid Horowitz-Myers conjecture}, J. High Energy Phys., \textbf{104} (2017). 

\end{thebibliography}
\end{document}